\documentclass[11pt]{article}
\usepackage{amsmath,amsfonts,amsthm,amscd,amssymb,graphicx}
\usepackage{color}

\numberwithin{equation}{section}

\newtheorem{theorem}{Theorem}[section]

\newtheorem{proposition}[theorem]{Proposition}

\newtheorem{corollary}[theorem]{Corollary}

\newtheorem{remark}[theorem]{Remark}

\def\bR  {\mathbb{R}}
\def\del{\partial }
\def \la {\langle}
\def \ra {\rangle}

\def\bR  {\mathbb{R}}

\def\del{\partial }
\def \la {\langle}
\def \ra {\rangle}

 \newcommand{\tO}{\tilde\Omega}
 \newcommand{\unn}{u_{\overline \nu}}
 \newcommand{\bee}{\begin{equation}}
 \newcommand{\eee}{\end{equation}}

\date{May 8, 2019}
\begin{document}

\title{Onsager's Conjecture with Physical Boundaries and an Application to the Vanishing Viscosity Limit}
\author{Claude Bardos\footnotemark[1]  \and  Edriss S.\ Titi\footnotemark[2]  \and Emil Wiedemann\footnotemark[3]}
\maketitle

\begin{abstract}
We consider the incompressible Euler equations in a bounded domain in three space dimensions. Recently, the first two authors proved Onsager's conjecture for bounded domains, i.e., that the energy of a solution to these equations is conserved provided the solution is H\"older continuous with exponent greater than 1/3, uniformly up to the boundary. In this contribution we relax this assumption, requiring only interior H\"older regularity and continuity of the normal component of the energy flux near the boundary. The significance of this improvement is given by the fact that our new condition is consistent with the possible formation of a Prandtl-type boundary layer in the vanishing viscosity limit.

\end{abstract}


\renewcommand{\thefootnote}{\fnsymbol{footnote}}

\footnotetext[1]{%
Laboratoire J.-L. Lions, BP187, 75252 Paris Cedex 05, France. Email:
claude.bardos@gmail.com}

\footnotetext[2]{%
Department of Mathematics,
                 Texas A\&M University, 3368 TAMU,
                 College Station, TX 77843-3368, USA. Department of Applied Mathematics and Theoretical Physics, University of Cambridge, Cambridge CB3 0WA, UK. Department of Computer Science and Applied Mathematics, Weizmann Institute of Science, Rehovot 76100, ISRAEL. Email: titi@math.tamu.edu \quad Edriss.Titi@damtp@cam.ac.uk}

\footnotetext[3]{%
Institut f\"ur Angewandte Analysis, Universit\"at Ulm, Helmholtzstra\ss e 18, 89081 Ulm, Germany. Email:
emil.wiedemann@uni-ulm.de}

\section{Introduction}
As early as in 1949, L.\ Onsager \cite{ON} conjectured that an ideal incompressible flow will conserve energy if it is H\"older continuous with exponent greater than 1/3. His conjecture, which was based on Kolmogorov's 1941 theory of turbulence and  was taken up by mathematicians only in the 1990s, when Eyink \cite{GEY} proved it for H\"older continuous functions with exponent larger than $1/2$, and Constantin-E-Titi \cite{CET} independently gave a complete proof of a ``stronger" version of the conjecture in the context of Besov spaces for exponent larger $1/3$. These results were later sharpened by Cheskidov et al.~\cite{CCFS}.

More recently, new interest has arisen in the relation between regularity and energy conservation as studied by Onsager. One direction of research has established the ``other direction" of Onsager's Conjecture, that is the optimality of the exponent 1/3. In other words, the aim has been to exhibit, for every $\alpha<1/3$, a weak solution of the Euler equations in $C^{0,\alpha}$ which does \emph{not} conserve energy. This has been achieved, as the culmination of a series of works by De Lellis-Sz\'ekelyhidi and others throughout several years, by Isett \cite{Isett} and Buckmaster-De Lellis-Sz\'ekelyhidi-Vicol \cite{BDSV}. However, one should keep in mind that the existence of solutions which belong to $C^{0,\alpha}$, with $\alpha<1/3$, and which dissipate the energy does not imply that all solutions that do not belong to $C^{0,\alpha}$, with $\alpha>1/3$, dissipate the energy. In fact, the authors of \cite{BT-Shear} provide  simple examples of weak solutions of the Euler equations which conserve the energy and which are not more regular than $L^2$, in particular they are not even bounded. Eventually, it is mostly in the presence of boundary effects that one can establish some type of complete equivalence between loss of regularity and non-conservation of energy, cf.\ Theorem 4.1 in \cite{BT}, following a theorem of Kato \cite{Ka}.

Another recent line of research has focused on the extension of the classical results \cite{GEY, CET, CCFS} to other systems of fluid dynamics, such as the inhomogeneous incompressible Euler and Navier-Stokes equations \cite{FGSW,LSh}, the isentropic compressible Euler equations \cite{FGSW}, the full Euler system \cite{DrEy}, the compressible Navier-Stokes equations \cite{Cheng Yu}, and a general class of hyperbolic conservation laws \cite{GMS}.

All these results are proved only in the absence of physical boundaries, i.e.\ on the whole space or the torus. Except for the case of the half-space \cite{RRS}, Onsager's Conjecture had not been studied in domains with boundaries until the recent work \cite{BT2} of the first two authors, who proved energy conservation of weak solutions of the incompressible Euler equations in (smooth) bounded domains $\Omega\subset\bR^n$ under the assumption that the solution be in $C^{0,\alpha}(\overline{\Omega})$ for some $\alpha>1/3$.

The aim of the present note is to give a less restrictive assumption on the regularity of the velocity. More precisely, we show that the energy is conserved if the weak solution $(u,p)$ of the Euler equations possesses  the following  properties (cf.~Theorem~\ref{CETloc}, below):
\begin{itemize}
\item At least for some $\beta<\infty$ and some $V_\gamma\subset\Omega$, where $\gamma>0$ and  $\{x\in\Omega: d(x,\partial\Omega)<\gamma\}\subset V_\gamma\,,$ one has $p\in L^{3/2}((0,T); H^{-\beta}(V_\gamma))\,.$
\item $u\in L^3((0,T);C^{0,\alpha}(\tilde{\Omega}))$ for any $\tilde{\Omega}\subset\subset\Omega$, with an exponent $\alpha>1/3$ that may depend on $\tilde{\Omega}$;
\item the energy flux
$$ \left(\frac{|u|^2}{2}+p\right)u$$
has a continuous normal component near the boundary of $\Omega$.
\end{itemize}

This may seem at first glance like a merely technical improvement, but, unlike the hypothesis of \cite{BT2}, our assumptions are consistent with the formation of a boundary layer in the vanishing viscosity limit. Indeed, consider a sequence of Leray-Hopf weak solutions of the Navier-Stokes equations, with no-slip boundary conditions, and  viscosity tending to zero. Then the discrepancy with the no-normal flow boundary condition for the Euler equations may lead to the formation of a boundary layer, where the normal directional derivative of the \emph{tangential} velocity component, and hence the $C^{0,\alpha}-$norm of the velocity, will blow up as the viscosity goes to zero. Note that this is not in contradiction with our regularity assumptions. The precise statement on the viscosity limit is contained in Theorem~\ref{noano0}.

As in \cite{BT2}, our argument relies on commutator estimates as introduced in \cite{CET}, but we pay special attention to clearly separate the \emph{local} and the \emph{global} arguments. In section~\ref{duchonrobert0}, we follow the work of Duchon-Robert \cite{DuchonRobert} to establish the local conservation of energy (Theorem~\ref{localduchonrobert}). However, since we no longer work in the whole space, or in the case of periodic boundary conditions,  we have to study more carefully the regularity of the pressure (Proposition \ref{lempress}). Section \ref{localtoglobal} presents the passage from local to global energy conservation. This is the only place where the regularity of the boundary of our domain comes into play. Finally, in section \ref{viscosity} we present the above-mentioned application concerning the vanishing viscosity limit.

Let us add a final remark concerning our assumptions: The hypothesis on the behavior of the pressure, near the boundary, is very weak and we don't see any way to remove it. This is because  in bounded domains the interior H\"older regularity of the pressure no longer automatically follows from that of the velocity. The H\"older spaces in the regularity assumption on the velocity, however, can presumably be replaced, e.g.,  by the critical Besov space from \cite{CCFS} or the Besov-$VMO$ condition from \cite{FW}; the main issue in such an improvement would be to transfer the elliptic estimates for the pressure from section~\ref{pressureest} to such Besov-type spaces. We prefer here to use H\"older spaces in order to avoid such difficulties and to keep the presentation simple.

\section{Weak solutions of the Euler equations defined on $(0,T)\times \Omega\,$}

%
We recall that with $\Omega$ denoting an open subset of $\bR^n$ a weak solution of the Euler equations   is a pair of distributions $(t,x)\mapsto (u(t,x),p(t,x)) \in (\mathcal D'((0,T) \times \Omega))^n\times \mathcal D'((0,T) \times \Omega)$ with $u \in L^q((0,T); L^2(\Omega)))^n$, for some $q\in [1,\infty]$,
which satisfies,  in the sense of distributions, the divergence free condition and the momentum equations:
\begin{equation}
\nabla\cdot u=0 \quad \hbox{in} \quad D'((0,T) \times \Omega) \quad \hbox{and} \quad   \del_t u + \nabla_x \cdot (u\otimes u) + \nabla_x p =0 \quad \hbox{in } (\mathcal D'((0,T)\times \Omega))^n 
\label{Eulerdistribution0}
 \end{equation}
 meaning, in particular, that for every $\Psi \in (\mathcal D ((0,T) \times \Omega))^n$
 \begin{equation}
   \la\la u, \partial_t\Psi   \ra\ra+ \la\la  u\otimes u, \nabla_x \Psi  \ra\ra + \la\la  p, \nabla_x \cdot \Psi \ra\ra=0\label{Eulerdistribution}\,,
  \end{equation}
 with $\la\la \cdot,\cdot \ra\ra$ denoting the  duality between $\mathcal D ' ((0,T) \times \Omega) $ and $\mathcal D ((0,T) \times \Omega\,)$.
 \begin{remark}
 \begin{itemize}
 \item Observe that the  condition $u \in L^q((0,T); L^2(\Omega))$ implies that $u\otimes u$ is well defined in  $D'((0,T)\times \Omega))\,.$

 \item Since $\mathcal D((0,T) \times \Omega)$ is the closure for the topology of  test functions of the tensor product $\mathcal D(0,T)\otimes \mathcal D(\Omega)$,  relation (\ref{Eulerdistribution0}) is equivalent to the relation:
 \begin{equation}
\forall \phi \in \mathcal D(\Omega) \,, \quad \del_t \la u ,\phi \ra + \la\nabla \cdot(u\otimes u ) ,\phi \ra +\la \nabla p,\phi \ra=0\,\label{reuler0}
\end{equation}
  in $\mathcal D'(0,T)\,,$ where  $\la\cdot,\cdot\ra $ denotes  the duality between $\mathcal D'(\Omega)$ and $\mathcal D (\Omega)\,.$

 \item Since $p$ is a distribution, it is locally of finite order (in $(x,t)$)  and therefore can be written as
 \begin{equation*}
 p(x,t) =(\frac {\del}{\del t})^k\nabla_x^l P(t,x)
 \end{equation*}
 with $k$ (resp.\ $l$) a finite integer (resp.\ finite multi-integer) and $P(t,x)\in L^\infty((0,T)  \times  \Omega).
 $
For both  the local result (cf.\ section \ref{duchonrobert0}) and  the global result, some extra regularity hypothesis of the pressure is required. With this assumption the impermeability boundary condition is not required for the local result. On the other hand, for global energy conservation (cf. section \ref {localtoglobal}), as expected, the impermeability boundary condition
\begin{equation}
u\cdot \vec n=0 \quad \hbox{on} \quad (0,T)\times \del\Omega \label{imp}
\end{equation}
is compulsory and, as usual, one observes that since $\nabla \cdot u =0$ in $(0,T)\times \Omega\,,$  relation (\ref{imp})  is   well defined at least in $C_{weak}((0,T);H^{-1/2} (\del \Omega))\,.$

 \end{itemize}
\end{remark}
\section{The local version of the Duchon-Robert Theorem \label{duchonrobert0}}

This section is devoted to the proof of a local energy conservation law in some time/space cylindrical domain
$$\tilde Q =(t_1,t_2)\times \tilde \Omega \subset\subset(0,T)\times \Omega.$$
No regularity hypothesis on $\Omega$ is required in this section other than being open and bounded. However it is assumed (without loss of generality) that  $\del \tilde \Omega$ is a $C^1$ manifold.

\begin{theorem} \label{localduchonrobert}
Let  $(u,p) \in L^q((0,T); L^2(\Omega  ))\times \mathcal D'((0,T)\times \Omega  )$, for some $q\in [1,\infty]$, be a weak solution of the Euler equations
\begin{equation}
\del_t u + \nabla \cdot (u\otimes u ) +\nabla p=0\,, \quad\nabla\cdot u=0,  \label{reuler}
\end{equation}
which in any open subset $\tilde Q = (t_1,t_2) \times \tilde \Omega \subset\subset(0,T)\times \Omega$ it satisfies the following two conditions:

\begin{enumerate}
\item ``Local in time regularity of the pressure near the boundary of $\tilde \Omega$".  For some $\gamma >0$ and for  $V_\gamma= \{x\in\tilde\Omega: d(x,\del\tilde \Omega)  <\gamma\}$  there exist $M_0(V_\gamma)>0$ and    $\beta(V_\gamma)>0$ such that
\begin{equation}
 p\in L^{3/2}((t_1,t_2); H^{-\beta(V_\gamma)}(V_\gamma))\le M_0(V_\gamma)<\infty;   \label{press0}
 \end{equation}

\item ``Interior $\frac13$ H\"older regularity": For some $\alpha(\tilde Q)>\frac 13$ and $M(\tilde Q)>0$ one has
\begin{equation}
\int_{t_1}^{t_2}  \|u(t,\cdot)\|^3_{  C^{0,\alpha({\tilde Q})} (\overline{\tilde \Omega})}dt \le M(\tilde Q)<\infty \label{local1}\,.
\end{equation}

\end{enumerate}
 Then $(u,p)$ satisfies in $ \tilde Q=(t_1,t_2)\times \tilde \Omega$ the local energy conservation:
\begin{equation}\label{localenergy}
\del_t\frac{|u|^2}2+\nabla_x\cdot\left(\left(\frac{|u|^2}2 +p\right)u\right)=0\quad \hbox{in }\, \mathcal  D'((t_1,t_2)\times  \tilde\Omega)\,.
\end{equation}
 \end{theorem}

 The proof of the above theorem is divided into three subsections. First, standard notations are introduced  and an extension-restriction Proposition is proven. Then an
 ``interior estimate"  for the pressure is deduced from   hypothesis (\ref{press0}). Eventually, the proof is accomplished by showing formula (\ref{localenergy}) for test functions of the form $\chi(t)\phi(x)$ with $\chi\in \mathcal D (t_1,t_2)$ and $\phi \in \mathcal D(\tilde \Omega)$, and then the proof  is extended by  the density of finitely many combinations of such test functions in $\mathcal D((t_1,t_2)\times \tilde\Omega)$.
 \subsection{Notations and extension-restriction construction}
For the sake of presentation convenience, we consider in this section, and in further treatments, standard mollifiers in the space $\bR_z^m$ as well as  a family of well adapted open sets and restriction or regularization of distributions. Consequently, the notations and properties will be adapted to  spatial domains or  ``time-space"    cylindrical domains, i.e.,
 $\bR_x^n$ or $\bR_{t,x}^{n+1}=\bR_t\times\bR_x^n\,.$ Let $\rho \in \mathcal D(\bR)$ be a non-negative function   with support in $\{ s \in \bR:\,|s|<1\}$ and  of total  integral $1$, i.e.,
 \begin{equation}
  \int_{\bR} \rho(s)ds =1\,.
 \end{equation}
We denote by $z\mapsto \rho_\sigma(z)$ the mollifier in $\bR^m_z$ given by
\begin{equation}
z\mapsto \rho_\sigma(z) = \frac 1{\sigma^m} \rho\left(\frac{|z|}{\sigma}\right)\,.
\end{equation}

For an open set $\tilde Q \subset\bR_z^m$ and a distribution $T\in \mathcal D'(\tilde Q)$, the relation $T=0$ is equivalent to the property that for any given test  function $\Psi\in\mathcal D (\tilde Q)$ (fixed for the rest of the argument) one has
\begin{equation}
\la\la T,\Psi\ra\ra =0\,.
\end{equation}
By definition $\Psi $ is compactly supported in $ \tilde{Q}\,.$ Its support will be denoted by $S_\Psi$, and for $\eta>0$, small enough, one can introduce three open sets with the following properties:
\begin{equation}\label{nestedsets}
\begin{aligned}
&S_\Psi\subset\subset Q_3\subset\subset Q_2\subset\subset Q_1 \subset\subset \tilde Q,\\
&d(S_\Psi, \bR^m\backslash Q_3)>\eta, \, d(Q_3, \bR^m\backslash Q_2)>\eta, \, d(Q_2, \bR^m\backslash Q_1)>\eta\,,  \hbox{ and}\,\,  d(Q_1, \bR^m\backslash \tilde Q)>\eta\,.
\end{aligned}
\end{equation}
Next, in order to extend to $\mathcal D'(\bR^m)$ distributions defined as elements of $\mathcal D'(\tilde Q)$, one introduces a smooth function $I_{2,\eta}\in \mathcal D(\bR_z^m)$ satisfying the following properties:
\begin{equation}
\hbox{for every}\,\, z\in Q_2 \,\, \hbox{one has}\,\, I_{2,\eta} (z)=1; \,\, \hbox{and for every} \,\, z\notin Q_1 \,\, \hbox{one has} \,\, I_{2,\eta}(z)=0\,. \label{extension00}
\end{equation}
As a consequence the support of $I_{2,\eta}$ is contained in $\overline{Q_1}$ and any  distribution $T\in \mathcal D'(\tilde Q)$ generates a distribution  $\mathcal D'(\bR^m_z)$ denoted $I_{2,\eta}T $ or $\overline T$ according to the formula:
\begin{equation}
 \la \la \overline T, \overline \Psi \ra\ra = \la \la I_{2,\eta}  T ,\overline \Psi \ra\ra= \la \la T, I_{2,\eta}\overline \Psi\ra\ra\,,
\end{equation}
for every $\overline\Psi \in \mathcal D(\bR^m_z)$.

For such a construction for a given $\Psi \in \mathcal D (\tilde Q)$, with support in $Q_3$,  the function $I_{2,\eta}\Psi$  belongs to $\mathcal D (\tilde Q)$ and it coincides with $\Psi$ in $\tilde Q$.  By an abuse of notation, with no risk of confusion, $I_{2,\eta}\Psi $ is identified with $\Psi$ and one has the following:

\begin{proposition} \label{extension-restriction}
 For any scalar or tensor  valued function $ w\in L^p(Q_1)$  (with $1\le p\le \infty$), and  for any $f\in C^\infty$, satisfying $f(0)=0$, one has the following properties:
$\overline {f} := I_{2,\eta} f$ given by the formula
\begin{equation}
\la\la\overline{f},\Phi\ra\ra = \la\la I_{2,\eta} f ,\Phi \ra \ra := \int_{\bR^m}  f(w(z)) I_{2,\eta} (z) \Phi(z)dz\,,
\end{equation}
for every $\Phi\in \mathcal D(\bR^m_z)$,
is a well defined distribution. Moreover, if  $\Psi\in \mathcal D(\bR^m_z)$,  with support in $Q_3\,,$ the following relation holds:
\begin{equation}\label{extension1}
 \la \la \overline f(w) , \Psi \ra\ra=  \int_{Q_3}  f(w(z))  I_{2,\eta}(z)\Psi(z) dz = \la\la  f(\overline w), \Psi\ra\ra\,.
\end{equation}
Furthermore, for any multi-order derivative $\,D^\alpha$ also the following relation holds:
\begin{equation}
 \la\la D^\alpha \overline{f(w)}, \Psi \ra\ra =\la\la D^\alpha f(\overline{w})), \Psi\ra\ra\,.\label{deriv0}
\end{equation}
Finally, for every $\sigma >0 $  small enough such that
\begin{equation}
0<\sigma<\frac\eta 2<\frac{d(Q_3,\bR_z^m \backslash \overline Q_ 2)}2\,,
\end{equation}
one has:
\begin{equation} \la \la \rho_\sigma \star \overline{f(w)},\Psi\ra	\ra=
\la \la \rho_\sigma \star f(\overline{w}),\Psi\ra	\ra\,. \label{rcommutte}
\end{equation}
\end{proposition}
\begin{proof}
Formula  (\ref{extension1}) is a direct consequence of the fact that on the support  of $\Psi$ one has $I_{2,\eta}=1\,.$
By the same token, for  formula (\ref{deriv0}) one returns to the definition of derivatives in the sense of distributions and writes:
  \begin{equation}
  \begin{aligned}
 \la\la D^\alpha \overline{f(w )}, \Psi \ra\ra  =(-1)^{|\alpha|}\la\la \overline{f(w )}, D^\alpha \Psi \ra\ra
  =(-1)^{|\alpha|}\la\la {f(w )}, I_{2,\eta} D^\alpha \Psi \ra\ra&
  \\
  =(-1)^{|\alpha|}\int_{\bR_z^m} f(I_{2,\eta}(z) w(z))  D^\alpha \Psi (z)dz=
 (-1)^{|\alpha|}\la\la  f( \overline{w  }), D^\alpha \Psi \ra\ra
  =  \la\la D^\alpha {f(\overline{w} )}, \Psi \ra\ra.  &
 \end{aligned}
 \end{equation}
 Eventually, for the proof of (\ref{rcommutte}), one observes that
 \begin{equation}
\hbox{whenever} \,\,d(z,Q_3)<\eta \,\, \hbox{then}\,\,   z\in Q_2 \,\, \hbox{and consequently}\,\,I_{2,\eta} (z)=1\,.
 \end{equation}
 As a result one has:
 \begin{equation}
 \begin{aligned}
& \la \la \rho_\sigma \star \overline{f(w)},\Psi\ra	\ra
 =\int_{\bR^m_z} f(w(z)) \Bigg( I_{2,\eta}(z) \int_{\bR^m_y}\rho_{\sigma}(z-y)  (y)\Psi(y)dy\Bigg)dz\\
 &=\int_{\bR^m_z} f(w(z)) \Bigg(\int_{\bR^m_y}\rho_{\sigma}(z-y)   \Psi(y)dy\Bigg)dz\\
 &=\int_{\bR^m_z} \Bigg( f(I_{2,\eta} (z) w(z))  \int_{\bR^m_y}\rho_{\sigma}(z-y)   \Psi(y)\Bigg)dy
 =\la \la \rho_\sigma \star f(\overline{ w}),\Psi\ra	\ra \,. \label{basic0}
 \end{aligned}
 \end{equation}
 \end{proof}
Below time-space cylindrical domains are considered. Let $\tilde Q =(t_1,t_2)\times \tilde \Omega \subset \subset(0,T)\times \Omega$.
For a given function $\psi\in \mathcal D(\tilde Q)$ with support contained in $(t_a,t_b) \times S_\psi$ we introduce  the following open sets satisfying:
$$
S_\psi\subset\subset \Omega_{3 } \subset\subset \Omega_{2 }\subset\subset \Omega_{1}\subset\subset \tilde \Omega\subset \bR_x^n\,.
$$
Consequently, one can choose  $\eta>0$, small enough, such that:
 \begin{equation}
 d(S_\psi, \bR^n\backslash \Omega_3)>\eta, \, d(\Omega_3, \bR^n\backslash \Omega_2)>\eta, \, d(\Omega_2, \bR^n\backslash \Omega_1)>\eta\,,  \hbox{ and}\,\,  d(\Omega_1, \bR^n\backslash \tilde \Omega)>\eta\,.
 \end{equation}
 For a time interval $(t_1,t_2)\subset (0, T)$,  one can choose  $\tau >0$ small enough, such that Proposition \ref{extension-restriction} will be applied to the open sets:
 \begin{equation}
 \begin{aligned}
 &(t_a,t_b)\times S_\psi \subset\subset Q_3= (t_1+3\tau, t_2-3\tau)\times \Omega_3\subset\subset Q_2= (t_1+2\tau, t_2-2\tau)\times \Omega_2\\
 &\subset\subset Q_1= (t_1+ \tau, t_2- \tau)\times \Omega_1\subset\subset \tilde Q= (t_1 , t_2 )\times \tilde \Omega\,. \label{openset}\\
\end{aligned}
 \end{equation}
 In the same way the extension process and notation are adapted as follows:
 One uses   functions  $I_{2,\tau}\in \mathcal D(\bR_t) $ , $I_{2,\eta}\in \mathcal D(\bR_x^n)$  and $I_{2,\sigma}$ with the following properties:
 \begin{equation}
 \begin{aligned}
& I_{2,\tau} (t)=1 \,\hbox{whenever}\,\, t\in (t_1+2\tau, t_2-2\tau);\, \hbox{and} \,\,
I_{2,\tau}(t)=0  \, \hbox{whenever}\,\,t\notin(t_1+\tau, t_2-\tau),\\
& I_{2,\eta} (x)=1 \,\hbox{whenever}\,\, x\in {\Omega_2};\, \hbox{and} \,\, I_{2,\eta}(x)=0 \, \hbox{whenever}\,\,   x\notin \Omega_1,\\
& I_{2,\sigma}(x,t) = I_{2,\tau} (t) I_{2,\eta}(x)\,.
\label{extension01}
\end{aligned}
 \end{equation}
 As above any distribution $T\in \mathcal D'(\tilde Q)$ is extended as a distribution in $\mathcal D'(\bR_t\times\bR^n_x)$ to
 \begin{equation}
 \overline T= I_{2,\sigma} T\,,
 \end{equation}
 and the same way the mollifiers $\rho_\sigma$ are replaced by the mollifiers:
 \begin{equation}
 \rho_\sigma(t,x)= \rho_{\kappa,\eta}(t,x)=\frac1{\kappa} \rho\left(\frac{|t|}{\kappa} \right)\frac 1{\epsilon^n} \rho\left(\frac{|x|}{\epsilon}\right),\label{mollifiers}
 \end{equation}
for every $\kappa\in(0,\frac\tau 2)$, and every  $\epsilon \in (0,\frac \eta 2)$.

 Finally,  for $w\in \mathcal D'(\bR_x^n)$ and for $W\in \mathcal D'(\bR_t\times \bR_x)$ we use the following notation:
 \begin{equation}
 \left(w\right)^\epsilon= \rho_\epsilon \star_x  w\quad\hbox{and} \quad \left(W\right)^{\epsilon,\kappa}=\rho_\sigma \star W = \rho_\kappa\star_t \rho_\epsilon \star_x W.
 \end{equation}

\subsection{Local estimate on the pressure}\label{pressureest}

With $\alpha$ and $\beta $ denoting the numbers $\alpha(\tilde Q)>\frac13$ and  $\beta(V_\gamma)\,$
  one has the following:

 \begin{proposition}\label{lempress}
 Let (u,p) be a weak solution of the Euler equations which satisfies in $\tilde Q=(t_1,t_2)\times \tO$ the hypothesis of Theorem  \ref{localduchonrobert}. Then the restriction of the pressure $p$ to $Q_2=(t_1,t_2)\times\Omega_{2}$ belongs to the space $L^{3/2}((t_1,t_2);C^{0,\alpha}(\overline{\Omega_{2}}))  $ and satisfies the estimate:
 \begin{equation}
  \|p\|_{L^{3/2}((t_1,t_2);C^{0,\alpha}(\overline{ {\Omega_{2}  }})) }
\le C(\|u\|_{ L^{3} ((t_1,t_2); C^{\alpha} (\tilde \Omega ))} ,\|p\|_{L^{3/2}((t_1,t_2); H^{-\beta}(V_\gamma))} )\,.\label{hypopressure}
 \end{equation}

 \end{proposition}
 \begin{proof} Taking the divergence of \eqref{Eulerdistribution0}, one deduces that $p$ satisfies in $\mathcal D'((0,T)\times \Omega)$ the relation
  \begin{equation}\label{poisson}
  -\Delta p =\sum_{i,j=1}^n\del^2_{x_i,x_j}\left(u_i u_j\right)\,. \\
  \end{equation}
	{One introduces a function $x\mapsto \tilde\eta(x) \in \mathcal D(\bR^n)$ with support in $  \tilde\Omega$   which is equal to $1$ in $\tilde \Omega \cap \{d(x,\del\Omega_1 >\frac \gamma 2\} $ and equal to $0$ for $x\in \tilde\Omega_1\cap \{d(x,\del\Omega_1)<\frac\gamma 4 \}\,.$
 In particular $\tilde \eta$ is equal to $1$ in a neighborhood of $\Omega_2$ and its gradient and Laplacian are equal to $0$ outside $V_\gamma$. Note that we may assume, without loss of generality, $\gamma<\frac\eta 4$.

 Then restricted to $\tilde\Omega$ the distribution $\tilde \eta p$ (which coincides with $p$ in $\Omega_1$) is a solution of the Dirichlet boundary value problem:
  \begin{equation}
  \begin{aligned}
& -\Delta (\tilde \eta p) =\tilde \eta \sum\del^2_{x_i,x_j} \left(u_i u_j\right)+R  \quad \text{in $\tilde\Omega$}
\\
&\quad \text{with $\tilde\eta p=0$ on $\del\tilde\Omega$ and with $R=-2 \nabla_x \tilde \eta \cdot \nabla_x p -p \Delta \tilde \eta$.}
\label{local}
\end{aligned}
 \end{equation}
Then estimate (\ref{hypopressure}) follows, by virtue of classical elliptic theory, from the H\"older assumption on $u$ and on the fact that $R$ is identically equal to $0$ in a neighborhood of $\Omega_2$.}


{ More precisely, we argue as follows:
First recall that $w:=\tilde \eta p$ as a distribution coincides with $p$ on $\Omega_1\,$, and that it is a solution of the Dirichlet problem~\eqref{local}.
 Writing $K_n(|x|)$ for the fundamental solution of the equation $-\Delta K_n= \delta_0$ in $\bR^n$,  $w$ can be expressed as follows:}
{
\begin{equation}
\begin{aligned}
&w= w_1+ w_2+w_3,\,\quad\text{where}\\
&w_1=  K_n\star \left[\tilde \eta \sum\del^2_{x_i,x_j}\left( u_i u_j\right)\right], \\
&w_2 = K_n\star R,
\end{aligned}
\end{equation}
and $w_3$ is the solution of
\begin{equation}
 \text{$-\Delta  w_3=0$ in $\tilde\Omega$ and $w_3=-w_1-w_2$ on $\del\tilde\Omega$.}
\end{equation}}
 {By virtue of standard H\"older elliptic regularity  estimates on the expression
 \begin{equation}
w_1 = K_n\star \left[\tilde \eta \sum\del^2_{x_i,x_j}\left(u_i u_j\right)\right](x) =\int_{\bR^n} K_n(|x-y|)(\tilde\eta \sum\del^2_{y_i,y_j}u_i u_j)dy\label{kril}
 \end{equation}
 (cf. \cite{Krylov}), one deduces from  (\ref{kril}) and (\ref {local1}) the relation:
 \begin{equation}
  \|w_1\|_{L^{3/2}((t_1,t_2);C^{0,\alpha}( { {\bR^n }})) }
\le C(\|u\|_{ L^{3} ((t_1,t_2); C^{\alpha} (\overline{\tilde \Omega }))} )\,.\label{hypopressure2}
 \end{equation}}

{ On the other hand, for $y\in \overline V_\gamma $ and $x\in \overline \Omega_1$, one has
 \begin{equation}
 |x-y|>\frac \eta 4.
 \end{equation}
 Since $\operatorname{supp}R \subset V_\gamma$,  by virtue of the standard elliptic estimates, this implies
 \begin{equation}
 \begin{aligned}
 & \|w_2\|_{L^{3/2}((t_1,t_2);C^{0,\alpha}(\overline{ {\Omega_{2}  }})) }
\le C( \|R\|_{L^{3/2}((t_1,t_2); H^{-\beta-1}(  V_\gamma))} )\le C( \|p\|_{L^{3/2}((t_1,t_2); H^{-\beta}( V_\gamma))} )\,.\label{hypopressure5}\\
  \end{aligned}
\end{equation}}
{With (\ref{hypopressure2}) and (\ref{hypopressure5}) one obtains eventually, for $w_3=-(w_1+w_2)$ on $\del\tilde\Omega$, the estimate
\begin{equation}
  \| w_3  \|_{L^{3/2}((t_1,t_2);C^{0,\alpha}(\overline{ { \Omega_{1}  }})) } 
 \le C(\|u\|_{ L^{3} ((t_1,t_2); C^{\alpha} (\overline{\tilde \Omega }))} + \|p\|_{L^{3/2}((t_1,t_2); H^{-\beta}( V_\gamma))})
\end{equation} Combining these estimates for $w_1+w_2+w_3$, restricted to $\Omega_2$, completes the proof.}



 \end{proof}
 From this proposition   one deduces the following:
 \begin{corollary} \label{corollary} Under the hypotheses of Proposition \ref{lempress}, the restriction of $\del_t u$ to  $Q_2=(t_1,t_2)\times \Omega_{2}$  is  bounded in
 $$L^{3/2}((t_1,t_2);H^{-1}( { {\Omega_{2}  }}))\,.$$
 \end{corollary}
 \begin{proof}
Thanks to  the relation
\begin{equation}
\del_t u =-\nabla \cdot ((u\otimes u) +  pI)
\end{equation}
the  proof follows from estimate (\ref{hypopressure}).
 \end{proof}

\subsection{Completion of the proof of Theorem \ref{localduchonrobert}}

Let us start   by considering a test function $\Psi=\chi(t)\phi(x)$ with compact support in $(t_a,t_b)\times S_\phi \subset \subset (t_1,t_2)\times \tilde \Omega$, and introduce,  for $ i =1,2,3$, the sufficiently small positive numbers $(\tau, \eta)\,,$ the open sets $Q_i$ satisfying relation (\ref{openset}) and the
corresponding mollifiers which satisfy (\ref{mollifiers}).  The function
$\chi(t)\phi(x)u(x,t)$ belongs to $L^3 (Q_1)$ and has support in $Q_3$. Therefore, one can introduce its extension by $0$ outside $Q_1:$
\begin{equation}
(t,x)\mapsto \overline{\left(\chi(t)\phi(x) u(t,x)\right)}= \chi(t)\phi(x) \overline{\left(u(t,x)\right)}\in L^3(\bR_t\times\bR_x^n)\,.
\end{equation}
This extension is  regularized according to the notation and formula:
  \begin{equation}
  \Psi_{\epsilon,\kappa} = \rho_{\epsilon,\kappa}\star\left (\chi(t)\phi(x) \left(\rho_{\epsilon,\kappa}\star\overline{u}\right)(t,x)\right) =:\left (\chi(t)\phi(x) \left(\overline{u(t,x)}\right)^{\epsilon,\kappa}\right) ^{\epsilon,\kappa} \in \mathcal D(\bR_t\times \bR_x^n)\,.
  \end{equation}
  With $\kappa\in (0, \frac\tau 2)$ and $\epsilon \in (0,\frac \eta 2)$,
  the support of $ \Psi_{\epsilon,\kappa}$ is contained in $Q_2\subset\subset \tilde Q$. Therefore, the formula
 \begin{equation}
 0=\la\la (\del_t u + \nabla_x \cdot (u\otimes u) +\nabla_x p ),\Psi_{\epsilon,\kappa}\ra\ra
 \end{equation}
 holds and is the sum of three well-defined terms:
 \begin{equation}
 I^{\epsilon,\kappa}_1=\la\la \del_t u, \Psi_{\epsilon,\kappa}\ra\ra\,,\quad I^{\epsilon,\kappa}_2=\la\la \nabla_x \cdot (u\otimes u), \Psi_{\epsilon,\kappa}\ra\ra\quad \hbox{and}\quad I^{\epsilon,\kappa}_3=\la\la \nabla_x p, \Psi_{\epsilon,\kappa}\ra\ra\,.
 \end{equation}
 The limit of these three terms for $(\epsilon,\kappa)\rightarrow 0$ is evaluated below,
 observing that the support of $\Psi_{\epsilon,\kappa}$ is contained in $Q_2\,.$
  Hence for the first term one has:
 \begin{equation}
 \begin{aligned}
 &I^{\epsilon,\kappa}_1=-\la\la u, \del_t \Psi_{\epsilon,\kappa}\ra\ra =-\int_{Q_2} u\cdot \del_t \Psi_{\epsilon,\kappa}dxdt
 =-\int_{\bR_t\times\bR^n_x} \overline {u}\cdot \del_t \Psi_{\epsilon,\kappa}dxdt\\
 &=
 -\int_{\bR_t\times\bR^n_x} \overline {u}\cdot \del_t\left(\chi(t)\phi(x) (\overline{u(t,x)})^{\epsilon,\kappa}\right)^{\epsilon,\kappa} dxdt\\
 &=\int_{\bR_t\times\bR^n_x} \left(\del_t\overline {u}\right)^{\epsilon,\kappa}\cdot  \left(\chi(t)\phi(x) (\overline{u(t,x)})^{\epsilon,\kappa}\right)dxdt\,.
 \end{aligned}
 \end{equation}
Since the support of $(\chi(t)\phi(x)\overline {u(t,x)})^{\epsilon,\kappa} $ is strictly contained in $Q_3\,,$ then by virtue of  the formula (\ref{deriv0}) of Proposition \ref{extension-restriction} we have:
\begin{equation}
 I^{\epsilon,\kappa}_1=\int_{\bR_t\times\bR^n_x} (\del_t(\overline {u})^{\epsilon,\kappa})\cdot (\overline {u(t,x)})^{\epsilon,\kappa}\cdot   \chi(t)\phi(x)   dxdt=-\la\la \frac{\left((\overline {u})^{\epsilon,\kappa}\right)^2}2 ,\del_t (\chi(t)\phi(x))  \ra\ra \label{I1} \,.
 \end{equation}
By the same token, for the second term, one has:
 \begin{equation}
 \begin{aligned}
 &I^{\epsilon,\kappa}_2=\la\la \nabla_x(u\otimes u), \Psi_{\epsilon,\kappa}\ra\ra
=-\int_{\bR_t}\!\int_{\bR^n_x}\left[(\overline{u}\otimes\overline{u})^{\epsilon,\kappa}:\nabla_x\left(\chi(t)\phi(x)( \overline {u(t,x)})^{\epsilon,\kappa}\right)\right]dxdt\\
 &=
 -\int_{\bR_t}\chi(t)\int_{\bR^n_x}\left[((\overline{u}\otimes\overline{u})^{\epsilon,\kappa}-
 (\overline{u})^{\epsilon,\kappa}\otimes(\overline{u})^{\epsilon,\kappa} ):\nabla_x(\phi(x)( \overline {u(t,x)})^{\epsilon,\kappa})\right]dxdt
 \\
 &-\int_{\bR_t}\chi(t)\int_{\bR^n_x}
 \left[(\overline{u})^{\epsilon,\kappa}\otimes(\overline{u})^{\epsilon,\kappa} :\nabla_x(\phi(x) (\overline {u(t,x)})^{\epsilon,\kappa})\right]dxdt\,.
 \end{aligned}
 \end{equation}
 On the other hand,   by a classical computation one has
 \begin{equation}
 \begin{aligned}
 &\int_{\bR_t} dt\int_{\bR^n_x}\left(\left(\overline{u})^{\epsilon,\kappa}\otimes(\overline{u})^{\epsilon,\kappa}\right) :\nabla_x(\chi(t)\phi(x)(\overline {u(t,x)})^{\epsilon,\kappa}\right)dx\\
 &=\int_{\bR_t}dt\int_{\bR^n_x}\frac{|(\overline{u})^{\epsilon,\kappa}|^2} 2  (\overline{u})^{\epsilon,\kappa}\cdot\nabla_x (\chi(t) \phi(x))dx\\
 &-\int_{\bR_t}\chi(t)dt\int_{\bR^n_x}\nabla_x \cdot (\overline {u(t,x)})^{\epsilon,\kappa}\frac{|(\overline{u})^{\epsilon,\kappa}|^2} 2  \phi(x) dx\,.
\end{aligned}
\end{equation}
Since $\chi(t)\frac{|(\overline{u})^{\epsilon,\kappa}|^2} 2  \phi(x)$ is a smooth function, with support contained in $Q_3\,,$
as above  (with   formula (\ref{deriv0}) of Proposition \ref{extension-restriction}) one also has:
\begin{equation}
\begin{aligned}
&\int_{\bR_t}dt\int_{\bR^n_x}\nabla_x \cdot (\overline {u(t,x)})^{\epsilon,\kappa}\chi(t)\frac{|(\overline{u})^{\epsilon,\kappa}|^2} 2  (\phi(x))dx=\\
&\int_{\bR_t}dt\int_{\bR^n_x} \left({\overline {\nabla_x \cdot u(t,x)}}\right)^{\epsilon,\kappa}\chi(t)\frac{|(\overline{u})^{\epsilon,\kappa}|^2} 2  (\phi(x))dx=0\,.
\end{aligned}
\end{equation}
Thus, one  eventually has:
\begin{equation}
\begin{aligned}
&I^{\epsilon,\kappa}_2=-\int_{\bR_t}\chi(t)\int_{\bR^n_x}\left[((\overline{u}\otimes\overline{u})^{\epsilon,\kappa}-
 (\overline{u})^{\epsilon,\kappa}\otimes(\overline{u})^{\epsilon,\kappa} ):\nabla_x(\phi(x) (\overline {u(t,x)})^{\epsilon,\kappa})\right]dxdt\\
&-\int_{\bR_t}  \chi(t) \int_{\bR^n_x}\frac{|(\overline{u})^{\epsilon,\kappa}|^2} 2  (\overline{u})^{\epsilon,\kappa} \cdot \nabla_x \phi(x) dxdt \label{I2}\,.
 \end{aligned}
\end{equation}
 For the   term $I_3^{\epsilon,\kappa}$ one uses the fact that (according to Proposition \ref{lempress} and Corollary \ref{corollary})
 \begin{equation}
 p\in {L^{3/2}((t_1,t_2);C^{0,\alpha}(\overline{ {\Omega_{2}  }})) }
 \end{equation}
and consequently one can also write:
 \begin{equation}
\begin{aligned}
&I^{\epsilon,\kappa}_3=\la\la \nabla p, \Psi^{\epsilon,\kappa}\ra\ra
=-\int_{\bR_t\times\bR^n_x} p  \nabla_x \cdot\left( \chi(t)\phi(x) ( (\overline {u(t,x)})^{\epsilon,\kappa})^{\epsilon,\kappa}\right)dxdt\\
&=-\int_{\bR_t}\chi(t)\int_{\bR_x^n}
 (\overline p)^{\epsilon,\kappa} (\overline {u(t,x)})^{\epsilon,\kappa}\cdot\nabla_x \phi(x)dxdt
 \\
 &-\int_{\bR_t}\chi(t)\int_{\bR_x^n}
 (\overline p)^{\epsilon,\kappa} \phi(x) \nabla_x\cdot (\overline {u(t,x)})^{\epsilon,\kappa}dxdt\,.
\end{aligned}
\end{equation}
Eventually,  as in the   previous two derivations,
\begin{equation}
\begin{aligned}
&\int_{\bR_t}\chi(t)\int_{\bR_x^n}
 (\overline p)^{\epsilon,\kappa} \phi(x) \nabla_x(\overline {u(t,x)})^{\epsilon,\kappa}dxdt\\
&=\int_{\bR_t}\chi(t)\int_{\bR_x^n}
 (\overline p)^{\epsilon,\kappa} \phi(x)({\overline{\nabla_x \cdot  u(t,x)}})^{\epsilon,\kappa}dxdt=0\,.
\end{aligned}
\end{equation}
Hence:
\begin{equation}
\begin{aligned}
I^{\epsilon,\kappa}_3=-\int_{\bR_t}\chi(t)\int_{\bR_x^n}
 (\overline p)^{\epsilon,\kappa} (\overline {u(t,x)})^{\epsilon,\kappa}\cdot \nabla_x \phi(x)dxdt\,. \label{I3}
 \end{aligned}
 \end{equation}
 With  formulas  (\ref{I1}), (\ref{I2})  and (\ref{I3}) for $I^{\epsilon,\kappa}_i\,,$ with $i=1,2$ and $3$, one obtains that:
 \begin{equation}
 \begin{aligned}
 &\int_{Q_2} \left[\frac{((\overline {u})^{\epsilon,\kappa})^2}2  \del_t (\phi(x)\chi(t))+\left( \frac{|(\overline{u})^{\epsilon,\kappa}|^2} 2 (\overline{u})^{\epsilon,\kappa} +(\overline p)^{\epsilon,\kappa} (\overline {u(t,x)})^{\epsilon,\kappa}\right)\cdot \nabla_x(\phi(x)\chi(t))\right]dxdt\\
 &=
 \int_{\bR_t}\chi(t)\int_{\bR^n_x}\left[((\overline{u}\otimes\overline{u})^{\epsilon,\kappa}-
 (\overline{u})^{\epsilon,\kappa}\otimes(\overline{u})^{\epsilon,\kappa} ):\nabla_x(\phi(x) \overline {u(t,x)})^{\epsilon,\kappa}\right]dxdt \label{final1}
 \end{aligned}
 \end{equation}

%
%
Now use the fact that the support of $(t,x)\mapsto (\chi(t)\phi(x))$ is contained in $Q_3\subset\subset (t_1,t_2)\times \Omega_2$ in conjunction with the following facts: With $C_k<\infty$ $(1\le k \le 3)$,
\begin{equation}
\begin{aligned}
&\hbox{by hypothesis}\quad  \|u\|_{ L^{3} ((t_1,t_2); C^{\alpha} (\overline{\Omega_2})) }\le C_1\,,\\
&\hbox{by Proposition \ref{lempress} }\quad \|p\|_{L^{3/2}((t_1,t_2);C^{0,\alpha}(\overline{ {\Omega_{2}  }})) }\le C_2\,,\\
&\hbox{by Corollary \ref{corollary}}\quad \|\del_t u\|_{ L^{3/2}((t_1,t_2);C^{0,\alpha}( \overline{\Omega_{2})} )} \le C_3\,;
 \end{aligned}
 \end{equation}
thus we can show, with the Aubin-Lions Theorem,  first that,  letting $\kappa\rightarrow 0$, in (\ref{final1}) one obtains the relation:
 \begin{equation}
 \begin{aligned}
 &\int_{Q_2} \frac{((\overline {u})^{\epsilon})^2}2  \del_t (\phi(x)\chi(t))+\left( \frac{|(\overline{u})^{\epsilon}|^2} 2 (\overline{u})^{\epsilon} +(\overline p)^{\epsilon} (\overline {u(t,x)})^{\epsilon}\right)\cdot \nabla_x(\phi(x)\chi(t))dxdt\\
 &=
 \int_{\bR_t}\chi(t)\int_{\bR^n_x}\left[\Big((\overline{u}\otimes\overline{u})^{\epsilon}-
 (\overline{u})^{\epsilon}\otimes(\overline{u})^{\epsilon}\Big) :\nabla(\phi(x) (\overline {u(t,x)})^{\epsilon})\right] dxdt \label{final4}
 \,.
 \end{aligned}
 \end{equation}
Second, using again the Aubin-Lions Theorem on the left-hand side of (\ref{final4}),  one has:
\begin{equation}
 \begin{aligned}
&\la\la \frac{|u|^2}{2} ,\del_t\chi(t)\phi(x)\ra\ra
  +\la\la \left(\frac{|u|^2}2 +p\right) u\cdot \nabla_x(\chi(t)\phi(x))\ra\ra\\
 &=\lim_{\epsilon\rightarrow 0} \int_{\bR_t}\chi(t)\int_{\bR^n_x}\left[\big((\overline{u}\otimes\overline{u})^{\epsilon}-
 (\overline{u})^{\epsilon}\otimes(\overline{u})^{\epsilon}\big) :\nabla_x(\phi(x) (\overline {u(t,x)})^{\epsilon})\right]dxdt
 \label{final5}
  \end{aligned}
 \end{equation}
For the right-hand side of (\ref{final5}), one observes that it satisfies the estimate
\begin{equation}
\begin{aligned}
&|\la\la
\chi(t),((\overline{u})^\epsilon \otimes (\overline{u})^\epsilon-(\overline{u  \otimes u})^\epsilon):\nabla(\phi(x) ({\overline u}^\epsilon))\ra\ra |\\
&\le  \int_{t_1}^{t_2} \chi(t) |\la ((\overline{u})^\epsilon \otimes (\overline{u})^\epsilon-(\overline{u  \otimes u})^\epsilon),\nabla (\phi(x)({\overline u}^\epsilon))\ra  |dt.\label{final2}
\end{aligned}
\end{equation}
For the right-hand side of (\ref{final2}),  following ideas that by now have become classical (cf. \cite{BT2},  \cite{CET}   or \cite{DuchonRobert}), we obtain the estimate:
\begin{equation}
\begin{aligned}
\int_{t_1}^{t_2}\chi(t) |\la ((\overline{u})^\epsilon \otimes (\overline{u})^\epsilon-&(\overline{u  \otimes u})^\epsilon),\nabla (\phi(x)({\overline u}^\epsilon))\ra  |dt\\
&\quad \quad\quad\le C(\chi, \phi) \epsilon^{3\alpha(\tilde Q)-1}\int_{t_1}^{t_2} \|u(t,\cdot)\|_{C^{0,\alpha(\tilde Q)}(\tilde \Omega)}^3dt.\label{final3}
\end{aligned}
\end{equation}
Letting $\epsilon\rightarrow 0$ and using  (\ref{final3}) we complete the proof of the theorem.
The proof of (\ref{final3}) is based on a time independent estimate which, for the sake of completeness,  is given below as the  object of the following:

\begin{proposition}Let $u \in C^{\alpha({\tilde \Omega})}(\overline{\tilde \Omega})$  with $\tilde \Omega \subset \subset \bR_x^n$ and $\alpha>\frac13$, $\phi\in \mathcal D(\tilde \Omega)$ a test function with support $S_\phi\,$, and for $ i=1,2,3$ select open sets
\begin{equation}
S_\phi \subset\subset \Omega_3\subset\subset \Omega_2\subset \subset \Omega_1\subset \subset\tilde \Omega,
\end{equation}
a number $\eta>0$, small enough, and a function $I_2\in \mathcal D(\bR_x^n)$ with  properties \eqref{nestedsets} and \eqref{extension00}, so that, in particular, $I_2$ is equal to $1$ in $\Omega_2$ and is equal to $0$ outside $ \Omega_1$. If $\rho_\epsilon$ is a space mollifier, then for the functions
\begin{equation}
 \left(\overline{u}\right)^\epsilon=\rho_\epsilon\star \overline u= \rho_\epsilon\star \left(I_2 u\right)\quad \hbox{and}\quad
 \left( \overline{u\otimes u}\right) ^\epsilon=\rho_\epsilon\star  \overline{u\otimes u}= \rho_\epsilon\star \left(I_2(u\otimes u) \right )\
\end{equation}
one has, for  $\epsilon \in (0,\frac \eta 2)\,,$ the estimate:
\begin{equation}
 |\la ((\overline{u})^\epsilon \otimes (\overline{u})^\epsilon-(\overline{u  \otimes u})^\epsilon):\nabla_x (\phi(x)({\overline u})^\epsilon )\ra  | \le C(\phi) \epsilon^{3\alpha(\tO)-1}(\|u\|_{C^{0,\alpha(\tO)}(\overline{\tilde \Omega})})^3.
\end{equation}
\end{proposition}
\begin{proof}
First use the formula
\begin{equation}
\begin{aligned}
&\nabla  (\phi(x)({\overline u}^\epsilon)) = \nabla\phi\otimes ({\overline u}^\epsilon)+\phi(x)\nabla({\overline u}^\epsilon) \\
&=\nabla\phi\otimes ({\overline u}^\epsilon)+\phi(x)\int_{\bR_x^n} \nabla_x\rho_\epsilon(x-y)I_2(y) u(y)dy\\
&=\nabla\phi\otimes ({\overline u}^\epsilon)+\phi(x)\int_{\bR_x^n} \nabla_y\rho_\epsilon(x-y) (I_2(x)u(x)-I_2(y)u(y))dy
\end{aligned}
\end{equation}
to show that
\begin{equation}
 |\nabla \cdot (\phi(x)({\overline u}^\epsilon))|
 \le C(\phi)\epsilon^{\alpha(\tO)-1}\|u\|_{C^{0,\alpha({\tilde \Omega)}}(\overline{\tO})}.\label{sup1}
\end{equation}
Second, thanks to the formula (\ref{basic0}) of Proposition  \ref{extension-restriction},  in $\Omega_2$ one has $(\overline{u  \otimes u})^\epsilon =(\overline{u}  \otimes \overline{u})^\epsilon$ and therefore with the elementary identity
\begin{equation}
\begin{aligned}
&((\overline u)^\epsilon\otimes(\overline u)^\epsilon)-(  \overline{u}\otimes \overline{ u} )^\epsilon = (\overline u-(\overline u)^\epsilon)\otimes(\overline u-(\overline u)^\epsilon) -\int( \delta_y \overline u \otimes \delta _y \overline u)\rho_\epsilon(y)dy\\
\hbox{with}\quad  &\delta_y \overline u =\overline u(x-y)-\overline u(x),  \label{elem4}
\end{aligned}
\end{equation}
 one obtains
  \begin{equation}
|( {\overline{u\otimes u}})^\epsilon-((\overline u)^\epsilon\otimes(\overline u)^\epsilon)|_{L^\infty(\Omega_{2})}\le C (\|u\|_{C^{0,\alpha(\tO)}(\overline{\tilde \Omega)}})^2 \epsilon^{2\alpha(\tO)}\label{sup2}.
\end{equation}
With (\ref{sup1} ) and (\ref{sup2}) the proof is completed.
\end{proof}
\begin{remark}
Formula (\ref{elem4}) is an  illustration of the similitude and difference existing between weak convergence, statistical theory of turbulence and regularization.
{With superscript $\epsilon$ denoting weak limits or statistical averages} one has
\begin{equation}
 ((\overline u)^\epsilon\otimes(\overline u)^\epsilon)-(  \overline{u}\otimes \overline{ u} )^\epsilon = (\overline u-(\overline u)^\epsilon)\otimes(\overline u-(\overline u)^\epsilon) -\int( \delta_y \overline u \otimes \delta _y \overline u)\rho_\epsilon(y)dy\,, \label{rt}
\end{equation}
where the first expression in the right-hand side of (\ref{rt}) is the Reynolds stress tensor. On the other hand the presence of the term
$$
\int( \delta_y \overline u \otimes \delta _y \overline u)\rho_\epsilon(y)dy,
$$
is due to the fact that instead of a weak limit of a family of solutions, it is an average of the same function $\overline u$ which is involved. This type of regularization is present in the original proof of Leray  \cite{Leray}, in several type of $\alpha$-models \cite{NS-alpha1, NS-alpha2, Leray-alpha, mLeray-alpha}, or eventually in turbulence modelling,  for instance, in the contributions of Germano \cite{Germano}.
\end{remark}
\section{From local to global energy conservation}\label{localtoglobal}

To consider the  global conservation of energy, the impermeability boundary condition will be used.
Hence, we assume that the boundary $\del \Omega$ is a $C^1$ manifold with $\vec n(x)$ denoting the outward normal at any point of $\del \Omega$. We introduce the function and the set
 \begin{equation}
 d(x)= d(x,\del \Omega)=\inf_{y\in \del \Omega }|x-y|\ge 0 \,,\,\, \hbox{and}\,\, V_{\eta_0}=\{ x\in \Omega\,, d(x)<\eta_0\},
  \end{equation}
which have  the following properties:

For $\eta_0 >0 $, small enough,
 $d(x)_{|V_{\eta_0}}\in C^1(\overline{V_{\eta_0}} )\,.$ Furthermore,  for any $x \in V_{\eta_0}$ there exists a unique $\sigma(x)\in  \del \Omega$ such that
 $d(x)=|x-\sigma(x)|$, and moreover one also has:
 \begin{equation}
 \nabla_x d(x) = -\vec n(\sigma(x)), \quad  \hbox{ for every}\,\, x\in V_{\eta_0}\,.
 \end{equation}


\begin{theorem}\label{CETloc}

Let $(u,p) \in L^q((0,T); L^2(\Omega  ))\times \mathcal D'((0,T)\times \Omega  )$, for some $q\in [1,\infty]$, be  a weak solution of the Euler equations 
satisfying the following hypotheses:
\begin{enumerate}
\item For some $\eta_0>0$,
\begin{subequations}
\begin{equation}
p\in L^{3/2} ((0,T); H^{-\beta}(V_{\eta_0}))\,, \quad \hbox{with}\quad  \beta <\infty \,,\label{forlebesque1}
\end{equation}
\begin{equation}
\lim_{\eta \rightarrow 0} \,  \int_0^T \frac 1{ \eta }\int_{ \{x\in \Omega:\, \frac{ \eta}{4} < d(x) <\frac{\eta}{2}<\frac{\eta_0}{2}\}} \left|\left(\frac {|u|^2}2 + p\right) u(t,x)\cdot \vec n(\sigma( x))\right|dx dt=0\,; \label{forlebesgue2}
\end{equation}
\end{subequations}
\item For every open set $\tilde Q =(t_1,t_2)\times \tO \subset\subset (0,T)\times \Omega $ there exists $\alpha(\tilde Q)>1/3$ such that $u$ satisfies Hypothesis (\ref{local1}) of Theorem  \ref{localduchonrobert}:
\begin{equation}
 \int_{t_1}^{t_2}  \|u(t,\cdot)\|^3_{  C^{0,\alpha(\tilde Q)} (\overline{\tO})}dt \le M(\tilde Q) <\infty \,. \label{local2}
\end{equation}
%
\end{enumerate}
Then, $(u,p)$ globally conserves the energy, i.e.,
\begin{equation}\label{energy-two-times}
\|u(t_2)\|_{L^2(\Omega)}=\|u(t_1)\|_{L^2(\Omega)},
\end{equation}
for any $0 < t_1<t_2 < T$. Moreover, $u \in L^\infty((0,T); L^2(\Omega)) \cap C((0,T); L^2(\Omega))$.
\end{theorem}
\begin{proof}
Start with  any open subset $\tilde \Omega \subset \subset \Omega$ such that
\begin{equation}\label{passumption}
  \Omega\setminus\tO\subset\subset{V_{\eta_0}},
\end{equation}
and then introduce a smooth function $x\mapsto \theta(x) \in \mathcal D(\Omega  )$ equal to $1$  for  $d(x)\ge  \frac{\eta_0}2$.
If $\Omega'$ is a domain with $\tO\subset\subset\Omega'\subset\subset\Omega$ and $\Omega\setminus\Omega'\subset\subset V_{\eta_0}$, then from the property
\begin{equation}
 u\otimes u\in L^{3/2}((t_1,t_2); C^{0,\alpha ((t_1t_2)\times\Omega')}(\overline{\Omega' })),\label{lip1}
  \end{equation}
 by Hypothesis \eqref{forlebesque1}, \eqref{passumption}, and Proposition \ref{lempress} we deduce that
  \begin{equation}
 p \in L^{3/2} ((0,T); C^{0,\alpha}(\overline{\tilde\Omega} )). \label{lip2}
 \end{equation}
  Then, thanks to Theorem \ref{localduchonrobert}, one concludes that the relation
\begin{equation}
 \frac {d}{dt} \la \frac{|u|^2}2,  \psi  \ra - \la     \left(\frac{|u|^2}2 +p\right)u,\nabla_x\psi \ra=0 \label{global0}
\end{equation}
  holds  for any $\psi \in \mathcal D (\tilde\Omega) $ in the sense of $\mathcal D'(0,T)$.  Observe that estimates (\ref{lip1}) and  (\ref{lip2}) imply  that formula
  (\ref{global0}) remains also valid for test functions $\psi \in C_c^1(\tilde \Omega)$ (i.e., with compact support in $\tilde \Omega$).
  Eventually, we introduce a non-negative $C^\infty$ function, $s\mapsto \phi(s)$, which is equal to $1$ for $s>\frac12$ and is equal to $0$ for $s<\frac14\,.$ With $0<\tilde\eta<\eta_0$ one has:
 \begin{equation}
 \begin{aligned}
 &\psi_{\tilde \eta }(x)=\phi(\frac {d(x)}{\tilde \eta}) \in C^1 (\Omega)
 \\
 &\nabla_x \psi_{\tilde \eta } (x)= - \frac 1{\tilde \eta } \phi'\left(\frac{d(x)}{\tilde \eta }\right)\vec n(\sigma(x))\,, \, \, \hbox{ for}\, \, \frac {\tilde \eta } 4<d(x)<\frac {\tilde \eta }2\,; \hbox{ otherwise} =0\,.
\end{aligned}
\end{equation}
Setting  $\psi =\psi_{\tilde \eta }$ in (\ref{global0}),  one has

\begin{equation}
\begin{aligned}
&\int_{\Omega} \frac{|u(t_2,x)|^2}2\phi  \left( \frac {d(x)} {\tilde \eta }\right)dx-\int_{\Omega} \frac{|u(t_1,x)|^2}2 \phi \left( \frac {d(x)} {\tilde \eta } \right)dx\\
 &\quad \quad \quad   =-\int_{t_1}^{t_2} \int_{\Omega} \left(\frac{|u|^2}{2}
  +  p\right)
u(t,x)\cdot \vec n(\sigma(x)) \frac 1{\tilde \eta } \phi' (\frac{d(x)}{\tilde \eta } )dxdt\,.\label{pretrivial}
\end{aligned}
\end{equation}
Next we will show that claim follows from  (\ref{pretrivial})  by letting $\eta\rightarrow 0\,.$  Specifically,  since $u$ belongs to $L^q((0,T);L^2(\Omega))$, for some $q\in [1,\infty]$,  one can choose a time $t_1\in(0,T)$ where $\|u(t_1)\|_{L^2(\Omega)}$ is finite. Thanks to (\ref{forlebesgue2}) one has
\begin{equation}
\begin{aligned}
&\lim_{\tilde\eta\rightarrow 0} \left|\int_{t_1}^{t_2} \int_{\Omega} \left(\frac{|u|^2}{2}
  +  p\right)
u(t,x)\cdot \vec n(\sigma(x)) \frac 1{\tilde \eta } \phi' (\frac{d(x)}{\tilde \eta } )dxdt\right| \le\\
&\lim_{\tilde\eta\rightarrow 0} \int_{0}^{T} \int_{\Omega}\left| \left(\frac{|u|^2}{2}
  +  p\right)
u(t,x)\cdot \vec n(\sigma(x)) \frac 1{\tilde \eta } \phi' (\frac{d(x)}{\tilde \eta } )\right|dxdt\\
&\le  \lim_{\tilde\eta \rightarrow 0} \, \Big( \int_0^T \int_{ \{x\in \Omega:\, \frac{\tilde \eta}{4} < d(x) <\frac{\tilde\eta}{2}<\frac{\eta_0}{2}\}} \frac 1{\tilde \eta } |\phi' (\frac{d(x)}{\tilde \eta } )| \left|\left(\frac {|u|^2}2 + p\right) u(t,x)\cdot \vec n(\sigma( x))\right| dx dt \Big) \\
&\le C \lim_{\tilde\eta \rightarrow 0} \,  \int_0^T \frac 1{\tilde \eta }\int_{ \{x\in \Omega:\, \frac{\tilde \eta}{4} < d(x) <\frac{\tilde\eta}{2}<\frac{\eta_0}{2}\}} \left|\left(\frac {|u|^2}2 + p\right) u(t,x)\cdot \vec n(\sigma( x))\right|dx dt=0\,.
\end{aligned}
\end{equation}
On the other hand,     since the integrands in the following terms are non-negative
\begin{equation}
\int_{\Omega} \frac{|u(t_2,x)|^2}2\phi  \left( \frac {d(x)} {\tilde \eta }\right)dx\quad{\hbox{and}}\quad\int_{\Omega} \frac{|u(t_1,x)|^2}2 \phi \left( \frac {d(x)} {\tilde \eta } \right)dx
\end{equation}
one can apply  the Lebesque Monotone Convergence Theorem to show that they converge to $\frac{1}{2} \|u(t_2)\|_{L^2(\Omega)}^2$ and $\frac{1}{2} \|u(t_1)\|_{L^2(\Omega)}^2$, respectively, as $\tilde \eta \to 0$. This completes the proof of (\ref{energy-two-times}).

As a consequence of (\ref{energy-two-times}) $\|u(t)\|_{L^2(\Omega)}$ is constant for all $t\in (0,T)$ and in particular,  $u\in L^\infty((0,T); L^2(\Omega))$. From the above estimates it is standard to show that $u \in C_{\hbox{weak}} ((0,T);L^2(\Omega))$. Combined with the fact that $\|u(t)\|_{L^2(\Omega)}$ continuous function (because it is constant) it follows that $u \in C((0,T);L^2(\Omega))$.
\end{proof}
\begin{remark}
Observe that if as in \cite{BT2} the solution $u\in L^\infty((0,T);C^{0,\alpha}(\overline{\Omega}))$, for some $\alpha>\frac{1}{3}$, then all the hypotheses of Theorem \ref{CETloc} are satisfied, because in this case the pressure is bounded (as it has been shown in \cite{BT2}), and consequently this  yields the result of \cite{BT2}.

\end{remark}

\begin{remark}
Local versions of energy/entropy conservation for compressible models and other conservation laws were proved, e.g., in \cite{BGGTW,DrEy,FGSW,GMS}. Following similar argument to the one presented above, the respective global versions of these results, in domains with physical boundaries, have been established in \cite{BGGTW}. Notably, in the incompressible case the special treatment presented here and in \cite{BT2} for the pressure  might not be accessible for the general compressible case or other conservation laws. Instead, the proof presented in \cite{BGGTW} assumes that all the state variables  satisfy the required spatial regularity assumptions.
\end{remark}

\section{An application to the vanishing viscosity limit}\label{viscosity}

The above discussion leads to the introduction of a   sufficient condition for non-anomalous energy dissipation in the vanishing viscosity limit. This condition seems relevant since it is not in contradiction with the presence of a Prandtl type boundary layer.

\begin{theorem}\label{noano0}
Let $u_\nu(t,x)$ be a family of Leray-Hopf weak solutions of the Navier-Stokes equations in $(0,T)\times \Omega$ \begin{equation*}
\begin{aligned}
&\del_t u_\nu +(u_\nu\cdot\nabla_x)  u_\nu-\nu \Delta u_\nu +\nabla p_\nu =0,\,\quad \nabla\cdot u_\nu=0,\,\\
&u_\nu(t,x)=0 \hbox{ on}\quad  (0,T)\times\del\Omega\,
\end{aligned}
\end{equation*}
with $\nu$ independent initial data $u_\nu(0,\cdot)=u_0\in L^2(\Omega)\,.$

Assume that on $(0,T)\times \Omega$ the family $u_\nu$ satisfies the  following hypotheses:
\begin{enumerate}
\item There exist, for some $\eta_0>0$ small enough,  an open subset
$V_{\eta_0}=\{ x\in \Omega\,, d(x)<\eta_0\}$, and $\beta<\infty$ (both being  independent of $\nu$) such that:
\begin{equation*}
\sup_\nu \|p_\nu\|_{L^{3/2}((0,T); H^{-\beta}(V_{\eta_0}))}<\infty;
\end{equation*}
\item For any $\tO\subset\subset \Omega$ there exists $\alpha=\alpha(\tO) >\frac13$ and a constant $M(\tO)$ such that for any $\nu>0$ one has:
\begin{equation}
\|u_\nu \|_{L^3((0,T);C^{0,\alpha} (\overline{\tO}))}\le M(\tO); \label{local3}
\end{equation}
\item
\begin{equation}
 \lim_{\eta \rightarrow 0} \,  \lim_{\nu \to 0}\int_0^T\frac 1{ \eta }\int_{ \{x\in \Omega:\, \frac{ \eta}{4} < d(x) <\frac{\eta}{2}<\frac{\eta_0}{2}\}}\left|\left(\frac {|u_\nu|^2}2 + p_\nu\right) u_\nu(t,x)\cdot \vec n(\sigma( x))\right|\, dx dt=0\,. \label{forlebesgue3}
\end{equation}
%
\end{enumerate}
Then (extracting a subsequence $\overline \nu$ if necessary) $\unn$ converges weak$-*$ in $L^\infty((0,T);L^2(\Omega)) $ to a function $\overline{u_\nu}\in C_{\mathrm weak}([0,T);L^2(\Omega))$ which is a weak solution of the Euler equations with the same initial data $u_0(\cdot)$  and which also satisfies the hypotheses of Theorem \ref{CETloc}. Moreover,  $\overline{u_\nu}$ conserves the energy and hence belongs to $C([0,T);L^2(\Omega))\,.$

Furthermore, there  is no anomalous energy dissipation in the vanishing  viscosity limit, i.e., for every $T^* \in (0,T)$ one has

\begin{equation}
\lim_{ \nu\rightarrow 0}  \nu \int_0^{T^*}\int_\Omega |\nabla_x u_{ \nu}(t,x)|^2dxdt=0. \label{noano}
\end{equation}
\end{theorem}
 \begin{proof}
Since $u_\nu$ are Leray-Hopf weak solutions of the Navier-Stokes equations they satisfy the energy inequality:
\begin{equation}\label{lerayhopf}
\quad \frac12\|u_\nu(t)\|^2_{L^2(\Omega)} +\nu\int_0^t\int_\Omega|\nabla_x u_\nu(t,x)|^2dxdt\le \frac12\|u_0\|^2_{L^2(\Omega)}\,, \quad  \hbox{for every} \quad t\in (0,T).
\end{equation}
From the  above  and the fact that $u_\nu$ are weak solutions of the Navier-Stokes equations one infers  the existence of a subsequence $u_{\nu_j}$ which converges weak$-*$ in $L^\infty((0,T); L^2(\Omega))$ to a function $\overline{u_\nu}\in C_{\mathrm weak}([0,T); L^2(\Omega))$. Observe that in this process the divergence free, the  impermeability condition  the initial data and the so called ``admissibility condition" are preserved for the limit function $\overline{u_\nu}$. Specifically, one has:
\begin{subequations}
\begin{equation}
\hbox{ In}  \quad  (0,T)\times  \Omega\,, \quad \nabla_x \cdot \overline{u_\nu}=0\,, \quad   \hbox{on}  \quad(t,x)\in  (0,T) \times  \del\Omega\,, \quad  \overline{u_\nu}\cdot n(x)=0;
\end{equation}
\begin{equation}
 \lim_{t\rightarrow 0+} \int_\Omega \overline{u_\nu}(t,x)\phi(x)dx=
\int_\Omega u_0(x)\phi(x)dx\,, \quad \hbox{for every} \quad \phi\in \mathcal D(\Omega) \label{wc}
\end{equation}
and
\begin{equation}
 \|\overline{u_\nu}(t)\|_{L^2(\Omega)}\le \|u_0\|_{L^2(\Omega)}, \quad \hbox{for a.e.} \quad  t\ge 0.\label{admiss}
\end{equation}
\end{subequations}
Observer that as a consequence of the fact that $\overline{u_\nu}\in C_{\mathrm weak}([0,T); L^2(\Omega))$ and of (\ref{admiss}) one infers that
\begin{equation}\label{continuity-at-zero}
\lim_{t \to 0+} \| \overline{u_\nu}(t) - u_0\|_{L^2(\Omega)} =0\,.
\end{equation}

Next, we observe that for any $\tilde \Omega \subset\subset\,\Omega$ , thanks to  (\ref{local3}), the family $u_\nu$  is bounded in $L^3((0,T);C^{0,\alpha} (\overline{\tO}))$. Therefore, the arguments presented in Proposition \ref{lempress} are valid for $p_\nu$ which together with (\ref{lerayhopf}) imply that  $\del_t u_{\nu}$ is bounded in $L^{3/2}((0,T);H^{-1}( { {\tO  }}))$. Therefore with the Aubin-Lions compactness theorem one shows that up to a subsequence, which will be also denoted by  $  u_{\nu_j}$, when restricted to the open set $(0,T)\times \tO$ converges \emph{strongly} in $L^2((0,T)\times \tO)\,.$ Applying this remark to  a family of exhausting open subsets
$$
\tO_k\subset\tO_{k+1}, \quad \hbox{such that} \quad \cup \tO_k = \Omega\,,$$
in place of $\tO$ one can
 extract  a diagonal subsesquence, also denoted by   $u_{\nu_j}$, which converges in $L^2((0,T)\times\Omega)$ to the limiting function $\overline{u_\nu}\,.$
Such strong convergence is enough to conclude  that $\overline {u_\nu}$ is a solution of the Euler equation in $(0,T)\times\Omega$.
From the above and the uniform with respect to $\nu$ hypothesis  one can conclude that $\overline {u_\nu}$ satisfies the hypotheses of Theorem \ref{localduchonrobert}. Therefore, by Theorem \ref{CETloc} if  $\overline{u_\nu}$ satisfies (\ref{forlebesgue2}) then it is a solution which conserves  energy. Thus, one needs to show that $\overline{u_\nu}$ satisfies
\begin{equation}\label{forlebesgue4}
 \lim_{\eta \rightarrow 0} \,  \int_0^T\frac 1{ \eta }\int_{ \{x\in \Omega:\, \frac{ \eta}{4} < d(x) <\frac{\eta}{2}<\frac{\eta_0}{2}\}}\left|\left(\frac {|\overline{u_\nu}|^2}2 + \overline{p_\nu}\right) \overline{u_\nu}(t,x)\cdot \vec n(\sigma( x))\right|\, dx dt =0\,.
\end{equation}
To establish this we will use assumption (\ref{forlebesgue3}). First, observe that there exists a subsequence, which will also be denoted by $u_{\nu_j}$, such that for every $\eta >0$, fixed, and for  every $x$ satisfying $\frac{\eta}{4}< d(x) <\frac{\eta}{2}$ one has
\begin{equation}
\begin{aligned}\label{Artzela0}
&\lim_{\nu_j \to 0}  \int_0^T\left|\left(\frac {|{u_{\nu_j}}|^2}2 + {p_{\nu_j}}\right) {u_{\nu_j}}(t,x)\cdot \vec n(\sigma( x))\right|\, dt
= \\
&\int_0^T\left|\left(\frac {|\overline{u_\nu}|^2}2 + \overline{p_\nu}\right) \overline{u_\nu}(t,x)\cdot \vec n(\sigma( x))\right|\, dt\,.
\end{aligned}
\end{equation}
 This is true because by virtue of assumption (\ref{local3}) and Proposition \ref{lempress} the following sequence
\[
\int_0^T\left|\left(\frac {|{u_\nu}|^2}2 + {p_\nu}\right) {u_\nu}(t,x)\cdot \vec n(\sigma( x))\right|\, dt
\]
is equicontinuous with respect to $x$, in every compact subset of $\Omega$. Hence (\ref{Artzela0}) follows by the  Arzela-Ascoli theorem after resorting to a diagonal subsequence.

Note that  (\ref{Artzela0}) gives
 \begin{equation}
\begin{aligned}\label{Artzela1}
& \lim_{{\nu_j} \to 0}  \int_0^T \frac 1{ \eta }\int_{ \{x\in \Omega:\, \frac{ \eta}{4} < d(x) <\frac{\eta}{2}<\frac{\eta_0}{2}\}}\left|\left(\frac {|{u_{\nu_j}}|^2}2 + {p_{\nu_j}}\right) {u_{\nu_j}}(t,x)\cdot \vec n(\sigma( x))\right|\, dx dt
= \\
&\int_0^T \frac 1{ \eta }\int_{ \{x\in \Omega:\, \frac{ \eta}{4} < d(x) <\frac{\eta}{2}<\frac{\eta_0}{2}\}}\left|\left(\frac {|\overline{u_\nu}|^2}2 + \overline{p_\nu}\right) \overline{u_\nu}(t,x)\cdot \vec n(\sigma( x))\right|\, dx dt\,.
\end{aligned}
\end{equation}
Thus, by taking the limit of both sides, as $\eta \to 0$, the left-hand side limit equals zero by assumption (\ref{forlebesgue3}) and this implies condition (\ref{forlebesgue4}).

Therefore, the conditions of   Theorem \ref{CETloc} are satisfied and consequently  $\|\overline{u_\nu}(t)\|_{L^2(\Omega)}$ is  constant over the interval $(0,T)$; and by virtue of (\ref{continuity-at-zero}) one has
\[
\|\overline{u_\nu}(t)\|_{L^2(\Omega)}= \|u_0\|_{L^2(\Omega)}\,,
\]
for every $t\in [0,T)$.

Finally, we prove (\ref{noano}) which we establish by contradiction. Indeed, suppose that  (\ref{noano}) is not correct, therefore,  there exist a time $T^{**}\in (0,T)$ and a
 subsequence of solutions, denoted by  $u_{\nu'}$,  such that
\begin{equation}\label{Contradiction-Assumption}
\lim_{\nu'\rightarrow 0}  \nu' \int_0^{T^{**}}\int_\Omega |\nabla_x u_{ \nu'}(t,x)|^2dxdt=d>0\,.
\end{equation}
Then by  applying the above arguments for the sequence  $u_{\nu'}$, instead of for  $u_{\nu}$, one can extract, as above, a subsequence    $u_{\nu'_j}$ which converges to a weak solution of Euler equations,  $\overline{u_{\nu'}}$, which preserves the energy over the interval $[0,T)$, i.e., $\|\overline{u_{\nu'}}(t)\|_{L^2(\Omega)}= \|u_0\|_{L^2(\Omega)}$
for every $t\in [0,T)$.   From the energy inequality for the Leray-Hopf weak solutions of the Navier-Stokes equations one has:
\begin{equation}
  \frac12 \|u_{\nu'_j}(t)\|_{L^2(\Omega)}^2 + \nu'_j \int_0^{t}\int_\Omega |\nabla_x u_{ \nu'_j}(t,x)|^2dxdt \le\frac12\|u_0\|_{L^2(\Omega)}^2\,.
\end{equation}
Since
\begin{equation}
\lim_{\nu'_j\rightarrow 0}\frac12 \|u_{\nu'_j}(t)\|_{L^2(\Omega)}^2=\frac12\|\overline{u_{\nu'}}(t)\|_{{L^2(\Omega)}}^2\,,
\end{equation}
for  every $t\in (0,T)\setminus E$, with $|E|=0$; and since $\|\overline{u_{\nu'}}(t)\|_{L^2(\Omega)}= \|u_0\|_{L^2(\Omega)}$, for every $t\in [0,T)$, one concludes that
$$
\lim_{{\nu'_j}\rightarrow 0}  {\nu'_j} \int_0^t\int_\Omega |\nabla_x u_{ {\nu'_j}}(t,x)|^2dxdt=0\,,
$$
for  every $t\in (0,T)\setminus E$. Choose $t^* \in (T^{**}, T)\setminus E$, from the above and (\ref{Contradiction-Assumption}) one has
$$
0 < d = \lim_{{\nu'_j}\rightarrow 0}  {\nu'_j} \int_0^{T^{**}}\int_\Omega |\nabla_x u_{ {\nu'_j}}(t,x)|^2dxdt \le \lim_{{\nu'_j}\rightarrow 0}  {\nu'_j} \int_0^{t^{*}}\int_\Omega |\nabla_x u_{ {\nu'_j}}(t,x)|^2dxdt=0\,,
$$
which leads to a contradiction.

 \end{proof}
{\section{Conclusion}}

In this section we make some comments clarifying the rational behind some of the assumptions that have been made in this contribution.

1.  {\it The hypothesis concerning  the pressure for energy conservation by the Euler flow:} In the absence of physical boundaries the estimates on the pressure follow  directly from the equation
\begin{equation}
-\Delta p= \sum_{ij} \del_{x_i}\left( u_j \del_{x_j} u_i\right) =\sum_{ij} \del_{x_i}\del_{x_j}\left (u_i u_j\right).
\end{equation}
This is no more the case in the presence of physical boundaries and some (very weak) hypotheses seem to be both natural and compulsory (cf. \cite{BT2}). In \cite{BT2} this matter has been  the subject of Proposition 1.2, and in the present contribution it is discussed in Proposition (\ref{lempress}).

2. {\it  The vanishing viscosity limit:} Observe that  hypotheses (\ref{press0}) (a very mild assumption of the behavior of the pressure near the boundary)  and (\ref{local1}) (local estimate on the velocity, away from the boundary)  do not seem to be enough in order to obtain conservation of energy in the zero viscosity limit. However, hypothesis (\ref{forlebesgue3}):
$$
\lim_{\eta \rightarrow 0} \,  \lim_{\nu \to 0}\int_0^T \frac 1{ \eta }\int_{ \{x\in \Omega:\, \frac{ \eta}{4} < d(x) <\frac{\eta}{2}<\frac{\eta_0}{2}\}}\left|\left(\frac {|u_\nu|^2}2 + p_\nu\right) u_\nu(t,x)\cdot \vec n(\sigma( x))\right|\, dx dt =0\,.
$$
seems to be compulsory.
This has  also been confirmed by the analysis made in \cite{CV}, where instead of such a hypothesis on the work done by the Bernoulli pressure at the boundary, some uniform (with respect to the viscosity, as $\nu\rightarrow 0$) regularity of the Navier-Stokes flow,  $u_\nu$, is assumed (cf. for instance formula (3.2) in \cite{CV}). This allows the authors to obtain the convergence to an {\it admissible weak solution} of Euler equations. On the other hand, it is not shown that such solution may conserve the total energy or may coincide with a Lipschitz solution with the same initial data and as shown in \cite{BSW} with the constructions of examples of admissible weak solutions which do not preserve the total energy.

Condition (\ref{forlebesgue3}) is naturally implied by a more explicit, but stronger, hypothesis  that we state below  for the sake of clarification.
There exist an  open set
$$V_\gamma=\{x\in \Omega \quad \hbox{with} \quad d(x,\del\Omega) <\gamma\}\,,$$
 a constant $M$ and a $\nu-$uniform modulus of continuity
 $$ s>0\mapsto \omega(s) >0\quad  \hbox{with}
\quad \lim_{s\rightarrow 0} \omega(s)=0\,,
 $$
such that one has
 \begin{equation}
 \begin{aligned}
&\forall  (t,x) \in (0,T)\times V_\gamma:  \quad |u_\nu(t,x)|+|p_\nu (t,x)| \le M\\
&|u_\nu(t,x)\cdot \vec n(\sigma (x))|\le \omega(d(x))\, , \quad \hbox{ for every }\quad x\in \Omega\cap V_\gamma\,.  \label{forboundary2}
\end{aligned}
\end{equation}
 Observe, however, that  condition (\ref{forboundary2}) and a fortiori  condition (\ref{forlebesgue3}) are fully compatible with the existence of a scenario,  in term of the Prandtl equations of the boundary layer, when such a scenario turns out to be valid.

\section*{Acknowledgements}
The authors would like to thank the anonymous referee for the constructive remarks that have contributed to the improvement of this revised version of this work. C.B.\ and E.S.T.\ would like to thank the ``South China Research Center for Applied Mathematics and Interdisciplinary Studies (CAMIS)", South China Normal University, for its warm hospitality where this work was completed.
The work of E.S.T.\ was supported in part by the ONR grant N00014-15-1-2333, and by the Einstein Stiftung/Foundation - Berlin, through the Einstein Visiting Fellow Program, and by the John Simon Guggenheim Memorial Foundation.


\begin{thebibliography}{99}

\bibitem{BGGTW} C.~Bardos, P.~Gwiazda, A.~\'Swierczewska-Gwiazda, E.S.~Titi and E.~Wiedemann, On the extension of Onsager's conjecture for general conservation laws, {\it Journal of Nonlinear Science}, (2018). DOI:10.1007/s00332-018-9496-4.


\bibitem{BT-Shear} C.~Bardos and E.~S.~Titi, Loss of smoothness and energy conserving rough weak solutions for the $3d$ Euler equations, {\it Discrete and Continuous Dynamical
Systems} Series S, \textbf{3(2)} (2010), 185--197.


\bibitem{BT} C.~Bardos and E.~S.~Titi,   Mathematics and turbulence: where do we stand? {\it J. Turbul.} {\bf 14} (2013), 42--76.

\bibitem{BT2} C.~Bardos and E.~S.~Titi, Onsager's Conjecture for the incompressible Euler equations in bounded domains, {\it Arch. Ration. Mech. Anal.} \textbf{228(1)} (2018), 197--207.

\bibitem{BSW}C.~ Bardos, L.~Sz\'ekelyhidi, Jr., and E.~Wiedemann, On the absence of uniqueness for the Euler equations: the effect of the boundary (Russian), {\it  Uspekhi Mat. Nauk } {\bf 69} (2014), 3--22; translation in {\it  Russian Math. Surveys} {\bf 69} (2014), no. 2, 189--207.

  \bibitem{BDSV}  T. Buckmaster, C. De Lellis, L. Sz\'ekelyhidi Jr., and V. Vicol,
    Onsager's conjecture for admissible weak solutions (2017), {\it Comm. Pure  Appl. Math} {\bf 73} (2019), 229--274.


\bibitem{NS-alpha1}  S.~Chen, C.~Foias, D.~Holm, E.~Olson, E.~S.~Titi and S.~Wynne,
A connection between Camassa--Holm equations and
turbulent flows in channels and pipes, {\it Physics of Fluids} {\bf 11} (1999), 2343--2353.

\bibitem{NS-alpha2} S.~Chen, C.~Foias, D.~Holm, E.~Olson, E.~S.~Titi and S.~Wynne,
The Camassa--Holm equations  and turbulence,  {\it  Physica D}
{\bf 133} (1999), 49--65.


\bibitem{CCFS} A.~Cheskidov, P.~Constantin, S.~Friedlander, and R.~Shvydkoy, Energy conservation and Onsager's conjecture for the Euler equations, {\it Nonlinearity} {\bf 21} (2008), 1233.

\bibitem{Leray-alpha}   A.~Cheskidov, D.~D.~Holm, E.~Olson, and E.~S.~Titi,
On a Leray-$\alpha$ model of turbulence, {\it Royal
Society London,  Proceedings,  Series A,  Mathematical, Physical \&
Engineering Sciences}  {\bf 461} (2005), 629--649.

  \bibitem{CET} P.~Constantin, W.~E, and E.S.~Titi, Onsager's Conjecture on the energy conservation for solutions of Euler's equation, {\it Comm. Math. Phys.} {\bf 165} (1994), 207--209.

  \bibitem{CV}  P.~Constantin and V.~Vicol, Remarks on high Reynolds numbers hydrodynamics and the inviscid limit, {\it Journal of Nonlinear Science} {\bf 28(2)} (2018), 711--724.

\bibitem{DrEy} T.~D.~Drivas and G.~L.~Eyink, An Onsager singularity theorem for turbulent solutions of compressible Euler equations, {\it Comm. Math. Phys.} {\bf 359} (2018), 733--763.

  \bibitem{DuchonRobert}  J.~Duchon and R.~Robert, Inertial energy dissipation for weak solutions of incompressible Euler and Navier-Stokes equations, {\it Nonlinearity} {\bf 13 } (2000), 249--255.

\bibitem {GEY} G.~L.~Eyink, Energy dissipation without viscosity in ideal hydrodynamics, I. Fourier analysis and local energy transfer, {\it Phys. D} {\bf 78} (1994), 222--240.

\bibitem{FGSW} E.~Feireisl, P.~Gwiazda, A.~\'Swierczewska-Gwiazda, and E.~Wiedemann, Regularity and energy conservation for the compressible Euler equations, {\it Arch. Ration. Mech. Anal.} {\bf 223} (2017), 1375--1395.

\bibitem{FW} U.~S.~Fjordholm and E.~Wiedemann, Statistical solutions and Onsager's conjecture,  {\it Phys.\ D} {\bf 376/377} (2018), 359--365.

\bibitem{Germano} M.~Germano, Differential filters of elliptic type,
{\it Phys. Fluids} {\bf 29} (1986), 1757--1758.


\bibitem{GMS} P.~Gwiazda, M.~Mich\'alek, and A.~\'Swierczewska-Gwiazda, A note on weak solutions of conservation laws and energy/entropy conservation (2017), {\it Arch. Ration. Mech. Anal.} {\bf 229} (2018), 1223--1238.

\bibitem{mLeray-alpha} A.~A.~Ilyin, E.~M.~Lunasin and E.S. Titi, A
modified-Leray$-\alpha$ subgrid scale model of turbulence,
 {\it Nonlinearity}  {\bf 19} (2006), 879--897.

\bibitem{Isett} P.~Isett, A proof of Onsager's conjecture, {\it Ann. of Math.},  \textbf{188:3} (2018), 1--93.

\bibitem{Ka} T.~Kato, Remarks on on zero viscosity limit for nonstationary Navier-Stokes flows with boundary, {\it Seminar on Nonlinear Partial Differential Equations (Berkeley, Calif., 1983), Math. Sci.
Res. Inst. Publ.,} Vol. 2, Springer, New York, 1984, pp. 85--98.

\bibitem{Krylov} N. V. Krylov, Lectures on Elliptic and Parabolic Equations in H\"older Spaces. Graduate Studies in Mathematics, 12. {\it American Mathematical Society}, 1996.

\bibitem{Leray} J. ~Leray, Sur le mouvement d'un liquide visqueux emplissant l'espace (French), {\it Acta Math. } {\bf 63} (1934), 193--248.

\bibitem{LSh} T.~M.~Leslie and R.~Shvydkoy, The energy balance relation for weak solutions of the density-dependent Navier-Stokes equations, {\it J. Differential Equations} {\bf 261} (2016), 3719--3733.

\bibitem{ON} L.~Onsager, Statistical hydrodynamics,  {\it Nuovo Cimento} {\bf 6} (1949), 279--287.

\bibitem{RRS} J.~C.~Robinson, J.~L.~Rodrigo, and J.~W.~D.~Skipper, Energy conservation in the 3D Euler equations on $\mathbb{T}^2\times \mathbb{R}_+$, {\it Asymptot. Anal.}, (to appear), (2018). arXiv:1611.00181.

\bibitem{Schwartz} L.~Schwartz, Th\'eorie des Distributions (French). Publications de l'Institut de Math\'ematique de l'Universit\'e de Strasbourg, No.~IX-X. {\it Hermann, Paris}, 1966.

\bibitem{Temam} R.~Temam, On the Euler equations of incompressible perfect fluids, {\it J. Functional Analysis} {\bf 20} (1975), 32--43.


 \bibitem{Cheng Yu} C. Yu, Energy conservation for the weak solutions of the compressible Navier-Stokes equations, {\it Arch. Ration. Mech. Anal.} {\bf 225} (2017), 1073--1087.

\end{thebibliography}
\end{document}